\documentclass[a4paper,11pt]{article}

\usepackage{amsmath}
\usepackage{amssymb}
\usepackage{amsthm}
\usepackage{graphicx}
\usepackage{enumerate}
\usepackage{mathrsfs}
\usepackage{bbm}
\usepackage{hyperref}
\usepackage{mathtools}
\usepackage{titlesec}
\usepackage{xcolor}
\usepackage{tocloft}

\newtheorem{theorem}{Theorem}[section]
\newtheorem{proposition}[theorem]{Proposition}
\newtheorem{lemma}[theorem]{Lemma}

\theoremstyle{definition}
\newtheorem{definition}[theorem]{Definition}
\newtheorem{example}[theorem]{Example}

\DeclareMathOperator\Ln{Ln}
\DeclareMathOperator\Arg{Arg}
\def\eps{\varepsilon}

\DeclarePairedDelimiter\floor{\lfloor}{\rfloor}

\begin{document}
	
\title{Ergodic approach to the study of boundedness of solutions of one type of the first-order semilinear difference equations}
\author{Andrii Chaikovskyi, Oleksandr Liubimov}
\date{September 14, 2025}
\maketitle

\begin{abstract}
\noindent We investigate the sufficient conditions for boundedness of one type of difference equations of the form $x(n+1)=ax(n)+f(x(n)) + y(n), \;\\  n\geq 1$ in critical case $|a|=1$. 

For this equation the following assumptions are introduced:

\begin{itemize}
    \item The function $f: \mathbb{C} \to \mathbb{C}$ and the input sequence $\{y(n)\}_{n\geq1}$ are assumed to be bounded.

    \item $\text{Re}\left(\overline{f(\rho \, e^{2\pi i \theta})} \; \cdot ae^{2\pi i \theta}\right)$ converges uniformly on $[0,1) \ni \theta$ to some real-valued function $\Phi(\theta)$ as $\rho \to +\infty$ .
\end{itemize}

\noindent Combining the celebrated results of the probability and ergodic theory together with the geometric consideration of the problem, we show that under fairly general conditions this type of semilinear difference equations has all the solutions bounded.

Subsequently we formulate the quantitative version of our theorem and give the example of its application. 

In addition, in the last section we discuss the conditions of our main result and provide the constructions, which highlight the importance of these conditions.
\end{abstract}

\noindent \textbf{Keywords:} Semilinear difference equation, Ergodic theory, Qualitative analysis, Discrete dynamical systems.

\vspace{10pt}

\noindent \textbf{AMS 2020 Subject Classification:} 39A45, 37A99, 11J71.

\tableofcontents

\section{Introduction and statement of the main result}

In our paper we work with the semilinear difference equation of the form

\begin{equation}\label{difference_equation}
x(n+1)=ax(n)+f(x(n)) + y(n), \; n\geq 1
\end{equation}

\noindent on the complex plane $\mathbb{C}$. The case of our interest is the so-called \textit{critical case}: $|a|=1$. This case is rather interesting for the analysis, because unlike in case $|a|\neq1$ multiplication by $a$ in the term $ax(n)$ does not impact on the absolute value $|x(n)|$ and, therefore, boundedness of solutions of the equation \eqref{difference_equation} depends on the behavior of $f$ and $y(n)$ in a more complicated way than in case $|a|\neq1$. For example, provided that the function $f:\mathbb{C} \to \mathbb{C}$ and sequence $\{y(n)\}_{n\geq1}$ are bounded, one may show that in case $|a|<1$ all the solutions of \eqref{difference_equation} are bounded, and in case $|a|>1$ equation \eqref{difference_equation} has at most one bounded solution.

Before proceeding to the main result of this paper, we would like to discuss the existing approaches to the studying of qualitative(and/or quantitative) properties of the difference equations of form \eqref{difference_equation}. In our opinion, one of the most natural ways to analyze the properties of the equation \eqref{difference_equation} is to treat this equation as the \textit{nonlinear Volterra difference equation}. Indeed, iterating the recurrence relation, one can rewrite the difference equation \eqref{difference_equation} in a following way:

$$x(n)= a^nx(1)+ \sum_{j=1}^{n-1}a^{n-j}f(x_j)+\sum_{j=1}^{n-1}a^{n-j} y(j), \ \ n\geq 1$$

\noindent and obtain the Volterra difference equation. This type of equations was well studied by other authors and a great variety of different methods were developed to understand the evolution of such equations. These methods include, in particular, \textit{method of Lyapunov functions}, \textit{resolvent method}, \textit{fixed point theory methods} and \textit{transform methods}. For the comprehensive guide to this subject we highly recommend to refer to the book of Raffoul \cite{Raffoul_volterra_difference_equations} and all the bibliography present there. 

The counterpart of the equation \eqref{difference_equation} in Banach space with operator coefficient, namely the equation $x(n+1)=(\mathcal{A}x)(n) + f(x(n))+y(n), n\geq 1$ or, generally speaking, the equation $x(n+1)=(\mathcal{A}x)(n) + f_n(x(n))+y(n), n\geq 1$ was also well studied by other authors. One of the highly developed methods that exist in this area is the \textit{method of maximal regularity}. For the details on this approach we strongly recommend to refer to the monograph of Agarwal, Cuevas and Lizama \cite{Agarwal_Cuevas_Lizama_regularity_difference_eq_banach_spaces}. 

It is important to mention that all the approaches listed above usually deal with the functions $f$, which have some very nice and predictable in some sense behavior. For example, the function $f$ is often assumed to be continuous or Lipschitz continuous, small in some sense, etc. Although we also assume the function $f$ to have some nice properties, we, on the other hand, we additionally permit it to partially have quite \text{"chaotic"} behavior. In order to handle the \text{"chaotic"} part of behavior of the function $f$ we will use various techniques of the ergodic and probability theory.

Now let us move to the main result of our paper. In case we deal with there exists some $\varphi \in [0,1)$ such that $a=e^{2\pi i\varphi}$. For real $t$ we set as usual $e(t) := e^{2\pi i t}$ and the equation \eqref{difference_equation} can be rewritten in the following way:

\begin{equation}\label{equation_with_rotation}
x(n+1)=e(\varphi)x(n)+f(x(n)) + y(n), \; n\geq 1
\end{equation}

The main result of our paper is the following theorem.

\begin{theorem}[\textbf{Main Theorem}]\label{main_theorem}
    If $\varphi$ is irrational, function $f: \mathbb{C} \to \mathbb{C}$ and sequence $\{y(n)\}_{n \geq 1} \subset \mathbb{C}$ are bounded and, in addition, the following conditions hold:

    \begin{enumerate}
        \item There exists a function $\Phi : [0,1)\to \mathbb{R}$ such that

        $$\sup_{\theta \in [0,1)} \left \vert \Phi(\theta)-\Re \left(\overline{f(\rho \, e(\theta))} \; \cdot e(\varphi + \theta)\right)\right\vert \to 0, \; \rho \to +\infty$$

        \noindent where $\Re(z)$ denotes the real part of $z \in \mathbb{C}$ as usual.
        
        \item $\Phi$ is Riemann integrable on $[0,1)$ or, equivalently, for Lebesgue almost all $\theta \in [0, 1)$ we have 
        
        $$\lim_{\delta \to 0} \Phi(\theta+\delta) = \Phi(\theta)$$

        \item 

        There exists some constant $D_0 \in \mathbb{N}$ such that for all $n_2>n_1\geq 1$ with $n_2 - n_1 \geq D_0$ we have 
        $$\frac{1}{n_2 - n_1}\left\vert\sum_{n=n_1}^{n_2 - 1}y(n)e(-n\varphi)\right\vert + \int_{[0,1)} \Phi(s) \; d s \;  \leq \beta$$

        where $\beta < 0$ is some real constant.
    \end{enumerate}

    \vspace{10pt}

    \noindent Then for any $x(1) \in \mathbb{C}$ the solution $\{x(n)\}_{n \geq 1}$ of the difference equation 

    $$x(n+1)=e(\varphi)x(n)+f(x(n)) + y(n), \; n\geq 1$$

    \noindent is bounded.
    
\end{theorem} 

\begin{example}
    In the equation \eqref{equation_with_rotation} let $\varphi$ be an arbitrary irrational number, $y(n):=\frac{1}{\sqrt{n}}e(n\varphi), n\geq1$ and let the function $f: \mathbb{C} \to \mathbb{C}$ be defined as follows:

    \begin{equation*}
    f(z):=\left\{
    \begin{array}{@{}rl@{}}
         0 & \text{if $z=0$} \\
         
        (-1 + i \sin (|\Arg z|^{10}) ) \cdot \frac{z}{|z|}e(\varphi) & \text{if $|z|\neq 0$ and $\Arg z \in [0,10^{-2025}]$} \\

        \left(\frac{1}{\log\log\log(2+|z|)} + i \sin (3^{|z|})\right)\cdot \frac{z}{|z|}e(\varphi) & \text{if $|z|\neq 0$ and $\Arg z \notin [0,10^{-2025}]$}
         
    \end{array}
    \right .
\end{equation*}

\noindent Here $\Arg z$ stands for the principal branch of the argument $\arg z$ and note that we choose the range of $\Arg z$ to be $[0, 2\pi)$. Theorem \ref{main_theorem} tells us that every solution of the equation \eqref{equation_with_rotation} is bounded, which is not really obvious.

\vspace{10pt}

Consult \hyperref[sec:quantitative_version_of_main_theorem]{section 4} for the quantitative version of theorem \ref{main_theorem} and consult \hyperref[sec:discussion_of_conditions_of_main_theorem]{section 6} for the discussion about importance of the conditions present in its formulation provided with all the appropriate counterexamples.
    
\end{example}

\section{Auxiliary results}

We will work with the circle $\mathbb{T} = \mathbb{R}/\mathbb{Z} = \{x+\mathbb{Z} \; | \; x\in \mathbb{R}\}$. In fact, $\mathbb{T}$ is the set of real numbers mod $1$. Topology on $\mathbb{T}$ is the quotient topology induced by the usual topology on $\mathbb{R}$. This topology can also be given by the metric 

$$d_{\mathbb{T}}(a, b) = \text{dist}(a+\mathbb{Z}, b+\mathbb{Z}):= \min_{m\in\mathbb{Z}}|a-b+m|$$

\noindent which is actually a length of the shortest arc on $\mathbb{T}$ with endpoints $a, b \mod 1$.

With this topology $\mathbb{T}$ becomes a compact abelian group. The Haar measure $m_{\mathbb{T}}$ on it is the measure on reals mod $1$ induced by the ordinary Lebesgue measure. We will denote $\lambda=m_{\mathbb{T}}$. $\lambda$ is a \textit{nonnegative regular normed Borel measure in $\mathbb{T}$}, that is, $\lambda(\mathbb{T})=1$ and $\lambda(E)=\sup\{\lambda(C) \; : \; C \subset E, \; C \; is \; closed\}=\inf\{\lambda(U) \; : \; E \subset U, \; U \; is \; open\}$ for all Borel sets $E$ in $\mathbb{T}$. In addition, $\lambda$ is a complete measure: all the subsets of  measure zero sets are measurable.

For the details of the theory of topological groups you may, for instance, refer to the book of Kuipers and Niederreiter \cite{Kuipers_Niederreiter_Uniform_Distribution} and to the book of Einsiedler and Ward \cite{Einsiedler_Ward_ergodic_theory}.

We remark that the functions $f:\mathbb{T} \to \mathbb{R}$ may be also considered as the functions $f:[0,1) \to \mathbb{R}$ which are $1$-periodically continued to the 1-periodic function $\tilde{f}:\mathbb{R} \to \mathbb{R}$. It means that for $x \in \mathbb{T}$ we have $f(x) = \tilde{f}(\{x+\mathbb{Z}\})=\tilde{f}(x \mod 1)$.

The following lemma is actually a continuous variant of the theorem of Egorov. The argument in its proof completely replicates the argument used in proof of the original theorem.

\begin{lemma}\label{almost_continuous_measurable_function_lemma}
    Let $y_0 \in \mathbb{R}$ and $f:\mathbb{T} \times\mathbb{R}_{\geq 0} \to \mathbb{R}$ be a function with the following properties:

    \begin{enumerate}
        \item For all $y$ in some neighborhood of $y_0$ function $f(\; \cdot \;,y):\mathbb{T} \to \mathbb{R}$ is measurable.
        
        \item For $\lambda$-almost all $x \in \mathbb{T}$ the limit $\lim_{y\to y_0} f(x, y) =: F(x)$ exists.
    \end{enumerate}

    \noindent Then for arbitrary $\varepsilon > 0$ there exists a set $A \subset \mathbb{T}$ with $\lambda(A) < \varepsilon$ such that 

    $$\sup_{x\in\mathbb{T}\setminus A} |f(x,y)-F(x)| \to 0, \; y\to y_0$$
    \end{lemma}

    Now we would like to provide a reader with the important notions and results from \textit{ergodic theory}. The wonderful reference to study this subject would be the book of Manfred Einsiedler and Thomas Ward \cite{Einsiedler_Ward_ergodic_theory}. The detailed proofs of all the statements written below can be found in this book.

    \begin{definition}
        Let $X$ be a compact Hausdorff space and $\mathfrak{B}(X)$ be a Borel $\sigma$-algebra of open sets of $X$. For measurable map $T:X\to X$ Borel measure $\mu$ on $X$ is said to be $T$\textit{-invariant} if for any $A \in \mathfrak{B}(X)$ we have $\mu(T^{-1}A)=\mu(A)$. In this case $(X, \mathfrak{B}(X), \mu, T)$ is called a \textit{measure-preserving system} and $T$ a \textit{measure-preserving transformation}.
    \end{definition}

    \begin{definition}
        A measure-preserving transformation $T:X \to X$ of a probability space $(X, \mathfrak{B}(X), \mu)$ is \textit{ergodic} if for any $A \in \mathfrak{B}(X)$:
        $$T^{-1}A=A \; \Longrightarrow \mu(A)=0 \; or \; \mu(A)=1$$
    \end{definition}

    \begin{proposition}\label{proposition_ergodic_rotation}
        The circle rotation $R_{\varphi}:\mathbb{T}\to \mathbb{T}, \; R_{\varphi}(x)=x+\varphi \mod 1$ is ergodic with respect to the Lebesgue measure $\lambda$ if and only if $\alpha$ is irrational.
    \end{proposition}

    \noindent By $\mathscr{L}^1_\mu(X)$ we denote the set of all measurable functions $g:X\to\mathbb{R}$ such that $\int_X|g|\;d\mu < \infty$. In this case $L^1$ norm of $g$ is defined as $\Vert g\Vert_1 := \int_X|g|\;d\mu$.
    For natural number $n$ notation $T^{n}$ stands for $n$ compositions of transform $T$, $T^{0}x:=x$

    \begin{theorem}[Maximal Ergodic Theorem]
        Consider the measure-preserving system $(X, \mathfrak{B}(X), \mu, T)$ on a probability space and $g$ a real-valued function in $\mathscr{L}^1_\mu(X)$. For any $\alpha \in \mathbb{R}$ define

        $$E_{\alpha}=\left\{x \in X \; : \; \sup_{N\geq1} \frac{1}{N} \sum_{n=0}^{N-1}g(T^{n}x) > \alpha\right\}$$

        \noindent Then

        $$\alpha\mu(E_{\alpha})\leq\int_{E_{\alpha}}g\;d\mu \leq \Vert g \Vert_1$$

        \noindent Moreover, $\alpha\mu(E_{\alpha} \cap A)\leq \int_{E_{\alpha} \cap A}g\;d\mu$ whenever $T^{-1}A=A$.
    \end{theorem}

    \vspace{10pt}

    \noindent This theorem allows us to control the amount of time orbits of points spend on a set of small measure. 

    \begin{lemma}\label{lemma_control_of_time_spent_on_set_of_small_measure}
        Let $E \subset X$ be a set with $\mu(E)=\varepsilon^2$. Then there exists a set $S \subset X$, with $\mu(S) \leq \varepsilon$ such that for all $N \geq 1$ and all $x \in X \setminus S$ we have 
        $$\#\{n, \; 0\leq n \leq N-1 \; : \; T^{n}x \in E\} =  \sum_{n=0}^{N-1} \mathbbm{1}_E(T^{n}x) \leq \varepsilon N$$
    \end{lemma}

    \begin{proof}
        Put $\alpha := \varepsilon$ and $g:=\mathbbm{1}_E$ in the Maximal Ergodic Theorem and use the inequality $\int_{E_{\alpha}} \mathbbm{1}_E \;d\mu \leq \int_{X} \mathbbm{1}_E \;d\mu = \mu(E)$.
    \end{proof}

    The next theorem is due to Birkhoff.

    \begin{theorem}[Pointwise Ergodic Theorem]\label{pointwise_ergodic_theorem}
        Let $(X, \mathfrak{B}(X), \mu, T)$ be a measure-preserving system on a probability space and $g \in \mathscr{L}^{1}_\mu (X)$. If $T$ is ergodic, then for $\mu$-almost all points $x\in X$ we have

        $$\lim_{N \to \infty} \frac{1}{N} \sum_{n=0}^{N-1} g(T^{n}x) = \int_{X}g \; d\mu$$
    \end{theorem}

	The next theorem provides us with the properties of the Riemann integrable functions, which will be useful in the proof of the main theorem.

    \begin{theorem}[Ergodic properties of a function]\label{ergodic_properties}
    		Let function $\Phi : [0,1)\to [0,1)$ be Riemann integrable on $[0,1)$ or, equivalently, for Lebesgue almost all $\theta \in [0, 1)$ we have 
    		
    		$$\lim_{\delta \to 0} \Phi(\theta+\delta) = \Phi(\theta).$$
    		
    		\noindent In addition, let
    		$$ \Omega:=\sup_{t\in [0,1)}|\Phi(t)|.$$
    		
    	If $\varepsilon \in (0, 1)$, then:
    	 	
    	1) there exists a set $S_1 \subset \mathbb{T}$ of small measure $\lambda(S_1) < \varepsilon^2$ and $\delta^{*} = \delta^{*}(\varepsilon) > 0$ such that for all $\delta \in (-\delta^{*}, \delta^{*})$ and all $\theta \in \mathbb{T} \setminus S_1$ we have
    	
    	\begin{equation}\label{pointwise_continuity_of_Phi}
    		\left\vert \Phi(\theta+\delta) -\Phi(\theta)\right\vert < \eps
    	\end{equation}
    	
    	2) there exists a set $S_2 \subset \mathbb{T}$ of big measure $\lambda(S_2) > 1 - \eps$ such that for all $N \geq 1$ and all $\theta \in S_2$ we have
    	
    	\begin{equation}\label{estimation_of_time_on_bad_set}
    		\#\{n, \; 0\leq n \leq N-1 \; : \; R^{n}_{\varphi}\theta \in S_1 \} \leq \eps N.
    	\end{equation}

		3)  there exists a set $S_3 \subset \mathbb{T}$ of big measure $\lambda(S_3) > 1 - \varepsilon$ and some $N_0$ such that for all $\theta \in S_3,\ N\geq N_0$ we have
		$$\left\vert \frac{1}{N}\sum_{n=0}^{N-1} \Phi(R_{\varphi}^{n}\theta) - \int_{\mathbb{T}} \Phi \; d \lambda\right\vert < \varepsilon.$$
			
		4) if $\alpha_0 \in S_2\cap S_3$ and $t_k \in \mathbb{T},\ 0 \leq k \leq N - 1$ are such that
		$$ d_\mathbb{T}\left(t_k, \; R_{\varphi}^k\alpha_0\right) < \delta^{*}(\eps), \; 0 \leq k \leq N - 1,$$
		then
		$$\left\vert\sum_{k=0}^{N-1} \Phi(t_k) - N\int_{\mathbb{T}} \Phi \; d \lambda\right\vert \leq 
		2N(\Omega+1)\eps.$$
    	\end{theorem}

    	\begin{proof}

		1. Using Lemma \ref{almost_continuous_measurable_function_lemma} for function $f(\theta, \delta) = \Phi(\theta+\delta)$  we obtain the desired statement.
		    	
    	\vspace{5pt}
    	
    	2. We use Lemma \ref{lemma_control_of_time_spent_on_set_of_small_measure}.

    	3. Using the Pointwise Ergodic Theorem \ref{pointwise_ergodic_theorem} for $T := R_{\varphi}$ and $g:=\Phi$ we obtain that for $\lambda$-almost all $\theta \in \mathbb{T}$ we have
        
    	$$\frac{1}{N}\sum_{n=0}^{N-1} \Phi(R_{\varphi}^{n}\theta) \to
    	\int_{\mathbb{T}} \Phi \; d \lambda,\ N\to\infty.$$
    	
    	Applying the theorem of Egorov to it, choose a set $S_3$ of big measure and some $N_0$ such that for all $\theta \in S_3,\ N\geq N_0$ we have
    	
    	\begin{equation}\label{approximation_of_ergodic_average}
    		\left\vert \frac{1}{N}\sum_{n=0}^{N-1} \Phi(R_{\varphi}^{n}\theta) - \int_{\mathbb{T}} \Phi \; d \lambda\right\vert < \varepsilon.
    	\end{equation}
    	
    	\vspace{5pt}

    	4.  By 2) we have 
    	$$\#\{k, \; 0\leq k \leq N - 1 \; : \; R^{k}_{\varphi}\alpha_0 \in S_1\} \leq \varepsilon N$$
    	
    	\noindent and, therefore, because of 1), we obtain
    	
    	$$\#\left\{k, \; 0\leq k \leq N - 1 \; : \; \left\vert \Phi(t_k) -\Phi(R^{k}_{\varphi}\alpha_0)\right\vert \geq \eps \right\} \leq \eps N$$
    	
    	\noindent which implies that
    	
    	$$\left\vert\sum_{k=0}^{N-1} \Phi(t_k) - \sum_{k=0}^{N-1} \Phi(R^{k}_{\varphi}\alpha_0)\right\vert \leq$$
    	$$\leq \sum_{k=0}^{N-1} \mathbbm{1}_{\left\{\left\vert\Phi(t_k) - \Phi(R^{k}_{\varphi}\alpha_0)\right\vert < \eps\right\}} \left\vert\Phi(t_k) - \Phi(R^{k}_{\varphi}\alpha_0)\right\vert + $$ 
    	$$ + \sum_{k=0}^{N-1} \mathbbm{1}_{\left\{\left\vert \Phi(t_k) -\Phi(R^{k}_{\varphi}\alpha_0)\right\vert \geq \eps\right\}}\left(\left\vert\Phi(t_k)\right\vert + \left\vert\Phi(R^{k}_{\varphi}\alpha_0)\right\vert\right) \leq \eps N + \eps N \cdot 2\Omega.$$

    	Because $\alpha_0 \in S_3$, we apply 3) and obtain
    	
    	$$\left\vert\sum_{k=0}^{N-1} \Phi(t_k) - N\int_{\mathbb{T}} \Phi \; d \lambda\right\vert \leq\left\vert\sum_{k=0}^{N-1} \Phi(t_k) - \sum_{n=0}^{N-1} \Phi(R_{\varphi}^{n}\theta)\right\vert +$$
    	$$+\left\vert \sum_{n=0}^{N-1} \Phi(R_{\varphi}^{n}\theta) - N\int_{\mathbb{T}} \Phi \; d \lambda\right\vert \leq 
    	N(\eps + 2\Omega\eps + \eps) = 2N(\Omega+1)\eps.$$
    	
    \end{proof}
    
    Also we need several inequalities for complex numbers.
    \noindent For convenience, we would like to introduce the following notation:
    $$\arg^{*}x := \frac{1}{2\pi} \arg{x},\ \Arg^{*}x:=\frac{1}{2\pi}\Arg{x}.$$
    
    \noindent \textbf{Remark.} We define the range of the principal branch value $\Arg^{*}{x}$ to be the interval $[0,1)$.
    
    \begin{lemma}\label{argument_change}
    	For all complex numbers $w_1,w_2 \neq 0$ we have
    	
    	$$d_\mathbb{T}(\arg^{*}w_1, \; \arg^{*}w_2) = \frac{1}{2\pi}\left\vert \Ln \frac{w_1}{w_2} - \ln\frac{|w_1|}{|w_2|} \right\vert$$
    \end{lemma}
    
    \noindent \textbf{Remark.} We define the range of the principal branch value $\Ln(z)$ of the complex logarithm $\ln(z)$ to be the strip $\{\Im(z) \in (-\pi, \pi]\}$, where $\Im(z)$ denotes the imaginary part of a complex number $z$ as usual.
    
    \begin{proof}
    	We write $w_1 = |w_1|e^{2\pi i \alpha_1}$, $w_2 = |w_2|e^{2\pi i \alpha_2}$, where $\alpha_1 := \arg^{*}w_1 , \alpha_2 := \arg^{*}w_2$.
    	
    	$$d_\mathbb{T}(\arg^{*}w_1, \; \arg^{*}w_2) = d_\mathbb{T}(\alpha_1, \alpha_2) = \min_{m \in \mathbb{Z}}{|(\alpha_1 - \alpha_2)-m|}=$$
    	$$=\min(\{\alpha_1 - \alpha_2\}, 1-\{\alpha_1 - \alpha_2\})  = \left\vert \left\{\alpha_1 - \alpha_2 + \frac{1}{2} \right\}-\frac{1}{2}\right\vert$$
    	
    	\noindent Here $\{x\}$ denotes the fractional part of the real number $x$. We conclude the proof of the proposition by noticing that 
    	
    	$$\left\{\alpha_1 - \alpha_2 + \frac{1}{2} \right\}-\frac{1}{2} = \frac{1}{2\pi i}\Ln e^{2\pi i(\alpha_1 - \alpha_2)} = \frac{1}{2\pi i}\Ln \frac{e^{2\pi i \alpha_1}}{e^{2\pi i \alpha_2}} = \frac{1}{2\pi i}\Ln\frac{(w_1 / |w_1|)}{(w_2/|w_2|)} = $$
    	
    	$$ = \frac{1}{2\pi i}\Ln\left(\frac{w_1}{w_2}\cdot\frac{|w_2|}{|w_1|}\right) = \frac{1}{2 \pi i}\left(\Ln\frac{w_1}{w_2} + \ln\frac{|w_2|}{|w_1|}\right) = \frac{1}{2 \pi i}\left(\Ln\frac{w_1}{w_2} - \ln\frac{|w_1|}{|w_2|}\right)$$

    \end{proof}
    
    \vspace{10pt}

    \begin{lemma}\label{lemma_estimation_rotation_caused_by_translation}
    	
    	For given $C>0$ and arbitrary complex numbers $w_1,w_2$ such that 
    	$$|w_1|\geq 2C,\ |w_1-w_2|\leq C$$
    	we have
    	$$d_{\mathbb{T}} (\arg^{*}w_2, \arg^{*}w_1) \leq \frac{C}{|w_1|}.$$
    \end{lemma} 
    
    \begin{proof}
    	We can see that $w_1\neq 0,\ w_2\neq 0,$ and $\left\vert\frac{|w_2|}{|w_1|}-1\right\vert \leq \frac{C}{|w_1|}$ since\linebreak	$||w_1| - |w_2||\leq |w_1 - w_2| \leq C.$
    	Using the previous lemma we have
    	$$d_{\mathbb{T}} (\arg^{*}w_2, \arg^{*}w_1) =  \frac{1}{2\pi}\left\vert \Ln \frac{w_2}{w_1} - \ln\frac{|w_2|}{|w_1|} \right\vert \leq \frac{1}{2\pi}\left(\left\vert \Ln \frac{w_2}{w_1}\right\vert + \left\vert\ln\frac{|w_2|}{|w_1|} \right\vert\right) =$$
    	$$=\frac{1}{2\pi}\left(\left\vert \Ln \left(1 + \frac{w_2-w_1}{w_1}\right)\right\vert + \left\vert\ln\left(1+\frac{|w_2|-|w_1|}{|w_1|}\right) \right\vert\right) \leq$$
    	$$\leq\frac{1}{2\pi}\left(\frac{3}{2}\frac{|w_2-w_1|}{|w_1|} + \frac{2|w_2-w_1|}{|w_1|}\right)
    	\leq\frac{1}{2\pi}\left(\frac{3C}{2|w_1|} + \frac{2C}{|w_1|}\right) < \frac{C}{|w_1|},$$
    	where we have used the well-known inequalities
    	$|\Ln(1+z)|\leq\frac{3}{2}|z|$ for $|z|\leq \frac{1}{2},$ and $\ln(1+x)\leq x$ for $x \in (-1,+\infty)$, and $\ln(1+x)\geq 2x$ for $x \in [-\frac{1}{2},0)$.
    \end{proof}
    
    \vspace{10pt}
   
    \begin{lemma}\label{absolute_value_change}
    	For given $C>0$ and all complex numbers $w_1,w_2 \neq 0$ such that $|w_1-w_2|\leq C$ we have 
    	$$|w_2|-|w_1| = \Re \left(\overline{w_2 - w_1} \; \cdot \; e(\arg^*{w_1})\right) + r(w_1)$$
    	\noindent where $|r(w_1)| \leq \frac{2C^2}{|w_1|}$.  
    \end{lemma}
    
    \begin{proof} 
    	We have
    	$$|w_2|-|w_1| = \frac{|w_2|^2 - |w_1|^2}{|w_2| + |w_1|} = \frac{w_2\overline{w_2} -  w_1\overline{w_1}}{|w_2| + |w_1|} =$$
    	$$= \frac{(w_2 - w_1)\overline{(w_2 - w_1)} + (w_2 - w_1)\overline{w_1} + \overline{(w_2 - w_1)}w_1}{|w_2| + |w_1|} =$$
    	$$= \frac{|w_2 - w_1|^2 + 2\Re(\overline{(w_2 - w_1)}w_1)}{|w_2| + |w_1|}.$$
    	
    	\noindent For convenience, denote $h := \frac{|w_2 - w_1|^2}{|w_2| + |w_1|}$ and observe that $|h| \leq \frac{C^2}{|w_1|}$. Hence
    	
    	$$\left\vert |w_2|-|w_1| - \Re \left(\overline{w_2 - w_1} \; \cdot \; e(\arg^*{w_1})\right) \right\vert  = \left\vert |w_2|-|w_1| - \frac{\Re \left(\overline{w_2 - w_1} \; \cdot \; w_1\right)}{|w_1|} \right\vert = $$
    	
    	$$ = \left\vert h + \Re \left(\overline{w_2 - w_1} \; \cdot \; w_1\right)\left(\frac{2}{|w_1| + |w_2|} - \frac{1}{|w_1|}\right) \right\vert \leq$$
    	
    	$$\leq \frac{C^2}{|w_1|} + |w_1- w_2|\cdot |w_1| \cdot \frac{|w_1 - w_2|}{|w_1|^2} \leq \frac{2C^2}{|w_1|}.$$
    	
    	\noindent Lemma follows.
    \end{proof}

\section{Proof of the Main Theorem}

At first we would like to describe the intuition behind the problem we are investigating in order to make our proof of the theorem \ref{main_theorem} more clear. 

\noindent Geometrically, the equation we work with can be interpreted in the following way:

\vspace{10pt}

Having found the point $x(n) \in \mathbb{C}$, in order to get the point $x(n+1)$ we, firstly, 

\begin{enumerate}
	\item Rotate point $x(n)$ counterclockwise around $0 \in \mathbb{C}$ by angle $2\pi \varphi$ and obtain point $e(\varphi)x(n)=e(\varphi +\frac{\arg{x(n)}}{2\pi})|x(n)|$.
	
	\noindent And then
	
	\item Move point $e(\varphi)x(n)$ along the vector $f(x(n)) + y(n) \in \mathbb{C}$ to obtain the point $x(n+1)$.
\end{enumerate}

\noindent Observe that step 1. does not change the absolute value $|x(n)|$, and rotates $\frac{\arg{x(n)}}{2\pi} \in \mathbb{T}$ by $\varphi$. At step 2. the absolute value may be changed(decreased or increased) and $\varphi + \frac{\arg{x(n)}}{2\pi} \in \mathbb{T}$ may be additionally rotated. 

In our approach we think of translation at step 2. as the composition of the rotation and the absolute value change(without rotation). Although rotations do not change the absolute value and we are solely concerned with the question of boundedness of sequence $\{x(n)\}_{n\geq1}$, they do affect the distribution of points $\frac{\arg{x(n)}}{2\pi}$ on the circle $\mathbb{T}$ and, therefore, make the process less controllable. As we will see later in this section, these rotations may be considered to be the \textit{"errors"} in a certain sense.

The main idea of our approach is based on the following intuition: the value $\Upsilon=\int_\mathbb{T} \Phi \; d\lambda$ may be regarded as the expected value of the random variable $\Phi : \mathbb{T} \to \mathbb{R}$ on the probability space $(\mathbb{T}, \mathfrak{B}(\mathbb{T}), \lambda)$. During the dynamical process described by the equation \eqref{equation_with_rotation}, the point $x(n)$ is rotated in the complex plane by an angle $2\pi \varphi$ around the origin $0$. Neglecting the rotation errors caused by the translations along the vectors $f$ and $y(n)$ we expect the absolute value $|x(n)|$ to be changed (approximately) by at most $\Upsilon + \frac{1}{N}\left\vert\sum_{n=1}^{N}y(n)e(-n\varphi)\right\vert$ on average. However, the rotation errors can't be neglected and in order to approach the problem we use the trick which we have called \textit{The Crane Trick}. Despite the fact that we are investigating the sufficient conditions for boundedness of solutions, we allow $|x(n)|$ to be large enough and, thus, making the behavior of $x(n)$ more clear for us. In other words, we \textit{lift} $x(n)$ at a large height where the \textit{rotation errors} are very small and the impact of the limit function $\Phi$ dominates these errors. Exactly as in the case when the construction crane rotates at a constant angular speed in windy weather with the suspended load: if the load is high, its dynamical behavior is more predictable.

\vspace{10pt}

Now we proceed to the proof of theorem \ref{main_theorem}.

\textbf{Proof.} We denote $\Upsilon := \int_{[0,1)} \Phi(s) \; d s \; $, $F:= \sup_{z\in \mathbb{C}}|f(z)|, Y:=\sup_{n\geq 1}|y(n)|$. We have $F, Y < +\infty$ since $f$ and $\{y(n)\}_{n\geq1}$ are assumed to be bounded. In addition, obviously, condition 3 implies that $F+Y > 0$.

Choose an arbitrary $\eps\in (0, \min(\frac{|\beta|}{16}, 1)).$ Using the statement 1) of theorem
\ref{ergodic_properties}, we choose the corresponding $\delta^* = \delta^*(\eps)$ with additional constraint $\delta^* <\frac{|\beta|}{8Y}.$

From 1) of theorem \ref{ergodic_properties} for $\eps_1 := \eps$ we obtain $\delta^*>0$ and the set $S_1$, that corresponds to this $\eps>0.$

Since $\varphi$ is irrational, sequence $\{n\varphi\ :\ n\geq 1\}$ is dense on the circle $\mathbb{T}$, so there exists a positive integer $N_d = N_d(\eps)$ such that distance from any point on circle to
some point of set $\{n\varphi\ :\ 1 \leq n < N_d\}$
is less than $\delta^*/2.$ The same is true for all the sets $\{\theta + n\varphi\ :\ 1 \leq n < N_d\}, \theta \in \mathbb{T}$. It is convenient to choose $N_d$ to be big enough such that $N_d \geq D_0+1$, where $D_0$ is as in the condition 3 of theorem \ref{main_theorem}.

 Finally, we take real "height" $H>0$ such that:
 \begin{equation}\label{h_constraints}
H \geq \frac{16(F+Y)^2}{|\beta|},\ H \geq \frac{2(F+Y)N_d}{\delta_*},\ H \geq |x(1)|,\ H \geq \rho(\eps),
\end{equation}
\noindent where $\rho(\varepsilon)$ is an arbitrary real number such that for every $\rho \geq \rho(\varepsilon)$ the inequality $\sup_{\theta \in [0,1)} \left \vert \Phi(\theta)-\Re \left(\overline{f(\rho \, e(\theta))} \; \cdot e(\varphi + \theta)\right)\right\vert \leq \varepsilon$ holds.

Let  
$$V := \left\{ z \in \mathbb{C} : H<|z|\leq H + (F+Y) \right\}.$$
We call the natural number $n \geq 2$ a \textit{visiting moment}, if $x(n) \in V$, but $x(n-1) \notin V$. Because $|f(z)|+|y(n)| \leq F+Y$ if solution is not bounded by $H$, at least one visiting moment exists.

Let $n_0\in\mathbb{N}$ be an arbitrary visiting moment and denote $\theta_0 := \arg^*x(n_0)$. Assume that for all $n \in \{n_0, n_0+1, \ldots, n_0+2N_d-1\}$ we have $|x(n)|\geq H$. In the other case we stop considering the \textit{visiting moment} $n_0$ and move to the next one.

There exists some point $\alpha_0\in S_2\cap S_3$ for which we can find $n_1>n_0$ with $1\leq n_1-n_0 \leq N_d$ such that $d_\mathbb{T} (\theta_0+(n_1-n_0)\varphi, \alpha_0) < \delta^{*} / 2.$

Applying lemma \ref{lemma_estimation_rotation_caused_by_translation}
with $w_1 = e(\varphi)x(n),\ w_2 = x(n+1), \\ C=F+Y,\ n_0\leq n < n_1$, we obtain the estimate 
$$d_\mathbb{T}(\arg^*(x(n+1)), \arg^*(x(n))+\varphi)\leq\frac{F+Y}{H},$$

\noindent and applying triangle inequality to it, we get
$$d_\mathbb{T}(\arg^*(x(n_1)), \arg^*(x(n_0))+(n_1-n_0)\varphi)\leq\frac{(F+Y)N_d}{H}  \leq \delta^{*} / 2.$$

\noindent which together with triangle inequality implies $d_\mathbb{T}(\arg^*(x(n_1)), \alpha_0) < \delta^{*}.$

Choose $n_2$ such that $n_1 < n_2 \leq n_1+N_d - 1$ with $n_2 - n_1 \geq D_0$. 
Next we consider the change of the absolute value of $x(n)$ during the following moments: $n_1, n_1+1, \ldots, n_2$. Recall that from our assumptions it follows that $|x(k)|\geq H$ for all $k\in \{n_1, n_1+1, \ldots, n_2\}$. Applying lemma \ref{absolute_value_change} to $w_2=x(n+1), \ w_1 = e(\varphi)x(n),\ C = F+Y$ we have

\begin{equation}\label{approximation_of_absolute_value_change_2}
	\begin{aligned}
    |x(n_2)| - |x(n_1)| = \sum_{k=n_1}^{n_2-1} (|x(k+1)| - |x(k)|) = \\ =\sum_{k=n_1}^{n_2-1}\Re \left(\overline{f(x(k))+y(k)} \cdot e(\varphi + \arg^{*}{x(k)})\right) + r_{\Re} \ ,
    \end{aligned}
\end{equation}

\noindent where $r_{\Re}$ is some real number such that $|r_{\Re}| < \frac{2(F+Y)^2(n_2 - n_1)}{H}$.

According to a condition 1 of theorem \ref{main_theorem} we have:

$$\sum_{k=n_1}^{n_2-1}\Re \left(\overline{f(x(k))} \cdot e(\varphi + \arg^{*}{x(k)})\right) = \sum_{k=n_1}^{n_2-1} \Phi(\arg^{*}{x(k})) + r_\Phi \ ,$$

\noindent where $r_\Phi$ is a real number such that $|r_\Phi|<\eps \cdot (n_2 - n_1)$.

By lemma \ref{lemma_estimation_rotation_caused_by_translation}
with $w_1 = e(\varphi)x(n),\ w_2 = x(n+1),\ C=F+Y,\ \\  n_1\leq n < n_2$ and because of the constraints for $H, n_2-n_1$ we obtain:

\begin{equation}\label{arg_x_is_near_good_set}
    	d_\mathbb{T}\left(\arg^{*}{x(k)}, \; R_{\varphi}^k\alpha_0\right) \leq \frac{(F+Y)(n_2 - n_1)}{H} < \delta^{*}, \; n_1 \leq k \leq n_2-1 \ .
\end{equation}

By the statement 4) of theorem \ref{ergodic_properties} we have

\begin{equation}\label{completed_approximation_of_sum_with_f}
    \sum_{k=n_1}^{n_2-1}\Re \left(\overline{f(x(k))} \cdot e(\varphi + \arg^{*}{x(k)})\right) = (n_2 - n_1)\Upsilon + r_f \ ,
\end{equation}

\noindent where $r_f$ is a real number such that $|r_f| \leq 2\eps(\Omega+1)(n_2-n_1)$.

Applying the estimate \eqref{arg_x_is_near_good_set}, we get

$$\left\vert\sum_{k=n_1}^{n_2-1}\Re \left(\overline{y(k)} \cdot e(\varphi + \arg^{*}{x(k)})\right) - \sum_{k=n_1}^{n_2-1}\Re \left(\overline{y(k)} \cdot e(\varphi + R_{\varphi}^{k}\alpha_0)\right)\right\vert \leq $$

$$\leq \sum_{k=n_1}^{n_2 - 1}\left\vert \overline{y(k)} \right\vert \cdot \left\vert e(\varphi + \arg^{*}{x(k)}) - e(\varphi + R_{\varphi}^{k}\alpha_0)\right\vert \leq $$

$$\leq Y\sum_{k=n_1}^{n_2 - 1} d_{\mathbb{T}}\left(\varphi + \arg^{*}{x(k)}, \; \varphi + R_{\varphi}^{k}\alpha_0\right)  \leq Y (n_2 - n_1) \delta^{*}$$

It means that
\begin{equation}\label{approximation_of_re_sum_y}
    \sum_{k=n_1}^{n_2-1}\Re \left(\overline{y(k)} \cdot e(\varphi + \arg^{*}{x(k)})\right) = \sum_{k=n_1}^{n_2-1}\Re \left(\overline{y(k)} \cdot e(\varphi + R_{\varphi}^{k}\alpha_0)\right) + r_y \,
\end{equation}

\noindent where $r_y$ is a some real number such that $|r_y|\leq Y (n_2 - n_1) \delta^{*}$.

\vspace{10pt}

\noindent Combine \eqref{approximation_of_absolute_value_change_2}, \eqref{completed_approximation_of_sum_with_f} and \eqref{approximation_of_re_sum_y} and deduce that

$$|x(n_2)| - |x(n_1)| = (n_2 - n_1) \Upsilon + \sum_{k=n_1}^{n_2}\Re \left(\overline{y(k)} \cdot e(R_{\varphi}^{k}\alpha_0)\right) + r \ ,$$

\noindent where $r$ is a real number such that 

$$|r| \leq \frac{2(F+Y)^2(n_2 - n_1)}{H} + 2\eps(n_2 - n_1) + Y\delta^*(n_2 - n_1).$$

\vspace{10pt}

According to the condition 3 of theorem \ref{main_theorem} we get

$$\sum_{k=n_1}^{n_2-1}\Re \left(\overline{y(k)} \cdot e(R_{\varphi}^{k}\alpha_0)\right) = \sum_{k=n_1}^{n_2 - 1} \Re  \left(\overline{y(k)} \cdot e(k\varphi) \; \cdot \; e(\alpha_0) \right) \leq$$ $$\leq\left\vert \sum_{k=n_1}^{n_2 - 1} \overline{y(k)} \cdot e(k\varphi) \right\vert \leq (\beta - \Upsilon) (n_2 - n_1). $$

Finally, according to the constraints for $\eps, \delta^*$ and $H$ we have
$$|x(n_2)| - |x(n_1)| \leq (n_2 - n_1) \Upsilon + (n_2 - n_1)(\beta - \Upsilon)  +$$
$$+ \frac{2(F+Y)^2(n_2 - n_1)}{H} + 2\eps(n_2 - n_1) + Y\delta^*(n_2 - n_1) \leq$$
$$\leq (n_2 - n_1)\left(\beta -\frac{\beta}{8} -\frac{\beta}{8} -\frac{\beta}{8}\right) = \frac{5\beta}{8}(n_2 - n_1) < 0$$

Hence, during the time $n\in \{n_1, n_1+1, \ldots, n_2\}$ the absolute value $|x(n)|$ has decreased provided that $|x(n)|\geq H$ for all $n \in \{n_1, n_1+1, \ldots, n_2\}$. Since $n_2 - n_1 \leq N_d-1$ and $|f(z)|+|y(n)| \leq F+Y$, we see that 

$$\max_{n\in \{n_1, n_1 +1, \ldots, n_2\}} (|x(n)|-|x(n_1)|) \leq \frac{1}{2}N_d(F+Y)$$

\noindent Therefore, for all $n\in \{n_1, n_1 +1, n_1+2, \ldots, n_2\}$ we have 
$$|x(n)| \leq |x(n_1)| + \frac{1}{2}N_d(F+Y) \leq |x(n_0)| +N_d(F+Y) + \frac{1}{2}N_d(F+Y) \leq $$

$$ \leq H+(F+Y)+N_d(F+Y) + \frac{1}{2}N_d(F+Y) \leq $$

$$ \leq H+\frac{1}{2}(3N_d+2)(F+Y)$$

The estimates above hold for any possible \textit{visiting moment} $n_0$.
We conclude that for all $n \geq 1$ we have 

$$|x(n)|\leq H+\frac{1}{2}(3N_d+2)(F+Y)$$ 

\noindent Therefore, the solution $\{x(n)\}_{n\geq 1}$ is bounded.

Theorem \ref{main_theorem} follows.

\section{Quantitative version of the Main Theorem}\label{sec:quantitative_version_of_main_theorem}

Theorem \ref{main_theorem} is qualitative in nature, however, the proof of this theorem given in the previous section allows us to state the quantitative version of it. 

\begin{definition}
    For $\delta> 0$, $m \in [0,1]$ and irrational $\alpha \in \mathbb{R}$ let $\mathcal{N}_\alpha(\delta, m)$ denote the least natural number $\mathcal{N}$ with the following property: for every Borel set $S \subset \mathbb{T}$ of measure $\lambda(S)\geq m$ and for every $\theta \in \mathbb{T}$ there exists some $n \in \{1,2,3\ldots, \mathcal{N}\}$ such that $\text{dist}_\mathbb{T}(\theta + n\alpha, S) < \delta$. Here $\text{dist}_\mathbb{T}(x,S) := \inf_{s \in S} d_{\mathbb{T}}(x,s)$ stands for the distance from point $x \in \mathbb{T}$ to the set $S$.
\end{definition} 

Note that $\mathcal{N}_\alpha(\delta, m)$ is well-defined, because $\alpha$ is irrational and the set $\{n\alpha \mod{1}\}_{n\geq 1}$ is dense in $\mathbb{T}$.

As in the proof, we denote $F:= \sup_{z\in \mathbb{C}}|f(z)|, Y:=\sup_{n\geq 1}|y(n)|$ and note that $F, Y < +\infty$. Initially, we chose an arbitrary $\varepsilon \in (0,\min(1,|\beta|/16))$ and took $\delta^{*}=\delta^{*}(\varepsilon)$ that corresponded to it in the statement of the theorem \ref{ergodic_properties}. It was convenient to take $\delta^{*}$ such that $\delta^{*} < \frac{|\beta|}{8Y}$. Theorem \ref{ergodic_properties} yielded the existence of  Borel sets $S_2, S_3 \subset \mathbb{T}$ of measure at least $1-\varepsilon$, which were crucial in our proof. Then we took the positive integer $N_d=N_d(\varepsilon)$ such that $N_d\geq D_0+1$. In addition, $N_d$ should have been such that for every $\theta_0 \in \mathbb{T}$ there existed some $n \in \{1,2,\ldots, N_d\}$ for which we had \\ $\text{dist}_{\mathbb{T}}(\theta_0 + n\varphi, S_2 \cap S_3) < \delta^{*}/2$. Since 

$$\lambda(S_2 \cap S_3)=\lambda(S_2)+\lambda(S_3)-\lambda(S_2 \cup S_3) \geq 2(1-\varepsilon)-1 = 1-2\varepsilon$$

\noindent we see that $N_d$ should also satisfy $N_d \geq \mathcal{N}_{\varphi}(\frac{\delta^{*}}{2}, 1-2\varepsilon)$. Next we chose the \text{"height"} $H$ to be an arbitrary positive real number such that 

\begin{equation}\label{h_constraints}
H \geq \frac{16(F+Y)^2}{|\beta|},\ H \geq \frac{2(F+Y)N_d}{\delta_*},\ H \geq |x(1)|,\ H \geq \rho(\eps),
\end{equation}

\noindent with $\rho(\varepsilon)$ being an arbitrary real number such that for every $\rho \geq \rho(\varepsilon)$ the  inequality $\sup_{\theta \in [0,1)} \left \vert \Phi(\theta)-\Re \left(\overline{f(\rho \, e(\theta))} \; \cdot e(\varphi + \theta)\right)\right\vert \leq \varepsilon$ holds.

At the end of the proof of theorem \ref{main_theorem} we established the estimate 

$$|x(n)|\leq H+\frac{1}{2}(3N_d+2)(F+Y)$$ 

We conclude this section by stating the quantitative version of theorem \ref{main_theorem} as follows.

\begin{theorem}[Quantitative version of the Main Theorem]\label{quantitative_version_of_main_theorem}
    Let $\varepsilon \in (0,1)$ and $\delta^*=\delta^{*}(\varepsilon)>0$ be as in theorem \ref{ergodic_properties}, but with the additional constraints $\varepsilon < \frac{|\beta|}{16}, \delta^* <\frac{|\beta|}{8Y}$. Then under the conditions of theorem \ref{main_theorem} we have the following estimate:

    $$|x(n)|\leq H+\frac{1}{2}(3N_d+2)(F+Y)$$

    \noindent where 

    $$F:=\sup_{z\in \mathbb{C}}|f(z)|, \ Y:=\sup_{n\geq 1}|y(n)|$$
    
    $$N_d :=\max{\left(D_0+1, \ \mathcal{N}_\varphi\left(\frac{\delta^*}{2}, 1-2\varepsilon\right)\right)}$$ 
    
    \noindent and

    $$H:=\max\left(\frac{16(F+Y)^2}{|\beta|}, \ \frac{2(F+Y)N_d}{\delta^*}, \  |x(1)|, \ \rho(\varepsilon)\right)$$

    \noindent with $\rho(\varepsilon)$ being an arbitrary real number such that for every $\rho \geq \rho(\varepsilon)$ the \\ inequality $\sup_{\theta \in [0,1)} \left \vert \Phi(\theta)-\Re \left(\overline{f(\rho \, e(\theta))} \; \cdot e(\varphi + \theta)\right)\right\vert \leq \varepsilon$ holds.
    
\end{theorem}

\section{Example of an application of quantitative version of the Main Theorem}

According to the condition 3 of theorem \ref{main_theorem} we see that the function $\Phi$ must satisfy $\int_{[0,1)}\Phi(s)ds < 0$, which seems to be quite restrictive. However, using the quantitative version of our theorem, we may deduce the boundedness of (possibly not all, but some) solutions of difference equation $x(n+1)=e(\varphi)x(n)+f(x(n))+y(n), n\geq 1$ in case  $\int_{[0,1)}\Phi(s)ds = 0$.

The example of such application of quantitative version of theorem \ref{main_theorem} follows.

\begin{theorem}\label{application_of_main_theorem}
    Let $\varphi \in [0,1)$ be irrational, bounded function $g : (0,+\infty)\to \mathbb{R}$ be such that $g(t)= t^{-\alpha}$ for all $t \geq M$, where $\alpha>0, M >0$ are some constants. Let function $f:\mathbb{C} \to \mathbb{C}$ be bounded and $\Re\left(f(z)\cdot e(-\varphi)\cdot \frac{\overline{z}}{|z|}\right) = -g(|z|)$ for all $z$ with $|z|>0$. In addition, let bounded sequence $\{y(n)\}_{n\geq1}$ be such that there exist some constants $K >0, \gamma>0$ so that for all $n_2 > n_1 \geq 1$ we have  
    
    $$\left\vert\sum\limits_{n=n_1}^{n_2-1}y(n)e(-n\varphi)\right\vert \leq K(n_2-n_1)^{1-\gamma}$$

    If $\alpha(1+\frac{1}{\gamma})<1$, then for all $x(1) \in \mathbb{C}$ the solution $\{x(n)\}_{n\geq 1}$ of the equation $x(n+1)=e(\varphi)x(n) + f(x(n))+y(n), n\geq 1$ is bounded.
\end{theorem}

\vspace{10pt}

\begin{proof}
    In proof we will use the standard asymptotic notation. For functions $g_1, g_2$ and $t_0 \in [0,+\infty]$ the notation $g_1(t) \ll g_2(t)$ or $g_2(t) \gg g_1(t)$ as $t \to t_0$ means that there exists a constant $C >0$ such that $|g(t_1)|\leq C|g_2(t)|$ for all $t$, which are sufficiently close to a point $t_0$ (in case $t_0 = +\infty$ it means that for all sufficiently large $t$). Notation $g_1(t) \asymp g_2(t), t\to t_0$ is a shorthand for $g_1(t) \ll g_2(t) \ll g_1(t), t\to t_0$.
    
    As before,  we denote $F:= \sup_{z\in \mathbb{C}}|f(z)|, Y:=\sup_{n\geq 1}|y(n)|$ and note that $F, Y < +\infty$.
    Take an arbitrary $x(1) \in \mathbb{C}$. For $\beta < 0$ let 
    $\delta^*(\beta):= \frac{|\beta|}{16(Y+1)}$ and $\varepsilon(\beta):= \min(\frac{|\beta|}{32}, \frac{\delta^*(\beta)}{4})$. Observe that we have $\mathcal{N}_{\varphi}\left(\frac{\delta^*(\beta)}{2}, 1-2\varepsilon(\beta)\right)=1$ for every $\beta < 0$.
    
    It is straightforward to see that for every $\beta \in(-1, 0)$ there exists the constant $D(\beta) \in \mathbb{N}$ such that for all $n_2 > n_1 \geq 1$ with $n_2 - n_1 \geq D(\beta)$ the inequality  $\frac{1}{n_2-n_1}\left\vert\sum\limits_{n=n_1}^{n_2-1}y(n)e(-n\varphi)\right\vert +2\beta \leq \beta$ holds. For instance, we may take $D(\beta) := \floor*{\left(\frac{K}{|\beta|}\right)^{1/\gamma}}+1$, where $\floor{\ \cdot \ }$ is the floor function.

    $$\text{Let} \ \ N_d(\beta) :=\max{\left(D(\beta)+1, \ \mathcal{N}_\varphi\left(\frac{\delta^*(\beta)}{2}, 1-2\varepsilon(\beta)\right)\right)} = D(\beta)+1 \ \ \text{and}$$
    
    $$H(\beta):= \max\left(\frac{16(F+Y)^2}{|\beta|}, \ \frac{2(F+Y)N_d(\beta)}{\delta^*(\beta)}, \ |x(1)|\right)$$

    \noindent We observe that $N_d(\beta) \asymp |\beta|^{-1/\gamma}$ and $H(\beta) \asymp |\beta|^{-1-\frac{1}{\gamma}}$ as $\beta \to 0-$ and, thus, we have $L(\beta):= H(\beta)+\frac{1}{2}(3N_d(\beta)+2)(F+Y) \asymp |\beta|^{-1-\frac{1}{\gamma}}, \beta \to 0-$ .
    
    Therefore, $g(L(\beta)) \asymp |\beta|^{\alpha(1+\frac{1}{\gamma})}, \beta \to 0-$. Since $\alpha(1+\frac{1}{\gamma})<1$ by hypothesis of theorem, one can find $\beta_1 < 0$ with the absolute value small enough such that $-g(L(\beta)) < 2\beta$ for all $\beta \in [\beta_1,0)$.


    By Lagrange's mean value theorem for all $ \beta < 0$, which are sufficiently close to zero (so that $H(\beta), L(\beta)) \geq M$), there exists some $\xi_{\beta} \in [H(\beta), L(\beta)]$ such that 
    
    $$|g(H(\beta)) - g(L(\beta))|=|g^{\prime}(\xi_\beta)|\cdot|(H(\beta) - L(\beta)|$$

    \noindent Because $|g^\prime(\xi_\beta)|=|\alpha|\cdot |\xi_\beta|^{-1-\alpha} \leq |\alpha|\cdot H(\beta)^{-1-\alpha} \ll |\beta|^{(1+\alpha)(1+\frac{1}{\gamma})}$ and $L(\beta)-H(\beta) \asymp |\beta|^{-\frac{1}{\gamma}}$ as $\beta \to 0-$, by the previous identity one obtains the estimate:

    $$|g(H(\beta)) - g(L(\beta))| \ll |\beta|^{(1+\alpha)(1+\frac{1}{\gamma}) - \frac{1}{\gamma}} = |\beta|^{1+\alpha(1+\frac{1}{\gamma})}, \beta \to 0-$$

    Since $1+\alpha(1+\frac{1}{\gamma}) > 1$ and $\varepsilon(\beta) \asymp |\beta|, \beta \to 0-$, we see that there exists some $\beta_{2} \in (-1,0)$ such that for all $\beta \in [\beta_2,0)$ and all $t \in [H(\beta), L(\beta)]$ one has $|g(t)-g(L(\beta))|\leq |g(H(\beta)) - g(L(\beta))|< \varepsilon(\beta)$.

    We set $\beta_0 := \max(\beta_1, \beta_2)$ and define the function $\tilde{g}:(0,+\infty)\to (0,+\infty)$ by $\tilde{g}(t):=g(t), \ \text{if $t < L(\beta_0)$}$ and $\tilde{g}(t):=g(L(\beta_0)), \ \text{if $t\geq L(\beta)$}$. Now we define the function $\tilde{f}:\mathbb{C}\to\mathbb{C}$ in a following way. For all $z \in \mathbb{C}, z\neq 0$ let
    
    $$\Re\left(\tilde{f}(z)\cdot e(-\varphi)\cdot \frac{\overline{z}}{|z|}\right) := -\tilde{g}(|z|) \ \ \text{and} \ \ \Im\left(\tilde{f}(z)\cdot e(-\varphi)\cdot \frac{\overline{z}}{|z|}\right) := \Im\left(f(z)\cdot e(-\varphi)\cdot \frac{\overline{z}}{|z|}\right)$$ 

    \noindent and let $\tilde{f}(0):=f(0)$.

    Consider the difference equation

    \begin{equation}\label{tilde_equation}
        \tilde{x}(n+1)=e(\varphi)\tilde{x}(n)+\tilde{f}(\tilde{x}(n))+y(n), n\geq 1, \tilde{x}(1)=x(1)
    \end{equation}

    \noindent It is easy to see that

    $$\sup_{\theta \in [0,1)} \left \vert \tilde{\Phi}(\theta)-\Re \left(\overline{\tilde{f}(\rho \, e(\theta))} \; \cdot e(\varphi + \theta)\right)\right\vert \to 0, \rho \to +\infty$$

    \noindent with $\tilde{\Phi}:[0,1)\to\mathbb{R}$ being a constant function $\tilde{\Phi} \equiv -g(L(\beta_0))$.

    Let

    $$\tilde{H}(\beta_0):= \max\left(\frac{16(F+Y)^2}{|\beta_0|}, \ \frac{2(F+Y)N_d(\beta)}{\delta^*(\beta_0)}, \ |x(1)|, \rho(\varepsilon(\beta_0))\right)=\max(H(\beta_0), \rho(\varepsilon(\beta_0))),$$

    \noindent where $\rho(\varepsilon(\beta_0))$ is an arbitrary real number such that for every $\rho \geq \rho(\varepsilon(\beta_0))$ the inequality $\sup_{\theta \in [0,1)} \left \vert \tilde{\Phi}(\theta)-\Re \left(\overline{\tilde{f}(\rho \, e(\theta))} \; \cdot e(\varphi + \theta)\right)\right\vert \leq \varepsilon(\beta_0)$ holds. 

    By the construction of $\tilde{f}$ and $\beta_0$ we may set $\rho(\varepsilon(\beta_0)):=\tilde{H}(\beta_0)$ and obtain that $\tilde{H}(\beta_0)=H(\beta_0)$. In addition, observe that for all $n_2 > n_1\geq 1$ with $n_2-n_1 \geq D(\beta_0)$ we have 

    $$\frac{1}{n_2-n_1}\left\vert\sum_{n=n_1}^{n_2-1} y(n)e(-n\varphi) \right\vert + \int_{[0,1)}\tilde{\Phi}(s)ds < \beta_0$$

    \noindent since $\int_{[0,1)}\tilde{\Phi}(s)ds = -g(L(\beta_0)) < -2\beta_0$.

    We apply the theorem \ref{quantitative_version_of_main_theorem} to the difference equation \eqref{tilde_equation} with 
    $\beta:=\beta_0$, $\varepsilon:= \varepsilon(\beta_0)$, $\delta^* := \delta^{*}(\beta_0)$, $\Phi := \tilde{\Phi}$, $\rho(\varepsilon):=\rho(\varepsilon(\beta_0))$, $D_0:=D(\beta_0)$, $N_d:=N_d(\beta_0)$ and deduce that for $x(1)$, chosen at the beginning of the proof, we have the estimate $|\tilde{x}(n)|\leq L(\beta_0), n\geq 1$. By the construction of $\tilde{f}$ we conclude that the solution $\{x(n)\}_{n\geq 1}$ of the difference equation $x(n+1)=e(\varphi)x(n)+f(x(n))+y(n), n\geq 1$ coincides with the solution $\{\tilde{x}(n)\}_{n\geq1}$ of the difference equation \eqref{tilde_equation}, that is, $x(n)=\tilde{x}(n), n\geq1$. Therefore, the solution $\{x(n)\}_{n\geq 1}$ is bounded.
\end{proof}

\begin{example}
    Suppose that $\varphi$ is irrational, $y(n):=0$ and $f:\mathbb{C} \to \mathbb{C}$ is the same as in the theorem \ref{application_of_main_theorem}, then this theorem tell us that for all $\gamma > 0$ and all $\alpha > 0$ with $\alpha(1+\frac{1}{\gamma}) < 1$ every solution of the difference equation $x(n+1)=e(\varphi)x(n)+f(x(n)), n\geq 1$ is bounded. Since we can take $\gamma > 0$ to be arbitrary large, it implies that this equation has all the solutions bounded for every $\alpha \in (0,1)$. 
\end{example}

\section{Discussion of the conditions of theorem \ref{main_theorem}}\label{sec:discussion_of_conditions_of_main_theorem}

In this section we would like to discuss the importance of the conditions of theorem \ref{main_theorem}. Before proceeding to  construction of the corresponding examples, we recall that by induction one can easily see that solution $\{x(n)\}_{n\geq1}$ of the difference equation $x(n+1)=e(\varphi)x(n)+f(x(n))+y(n), n\geq 1$ satisfies the following recurrence relation:

\begin{equation}\label{solution_recurrence_relation}
    x(n)=e(n\varphi)x(1) + e(n\varphi)\sum_{j=1}^{n-1}f(x(j))e(-j\varphi) + e(n\varphi)\sum_{j=1}^{n-1}y(j)e(-j\varphi), n\geq 2
\end{equation}

\subsection{Uniform convergence in the condition 1 of theorem \ref{main_theorem} can't be replaced with the pointwise convergence}

Let $\varphi \in [0,1)$ be an arbitrary irrational number. We can write $\mathbb{T}$ as the disjoint union of orbits $\mathcal{O}_\tau:=\{\tau+n\varphi : n \in \mathbb{Z}\}$, that is, $\mathbb{T} = \bigcup_{\tau} \mathcal{O}_\tau$ for orbits $\mathcal{O}_{\tau}$ with pairwise empty intersections. We construct a function $f$ by defining it on each orbit $\mathcal{O}_{\tau}$ in the following way.

$f(0):=1$. If $z \neq 0$, then there exists some unique $\tau$ such that $\arg^{*}z \in \mathcal{O}_{\tau}$, that is, $\arg^*z = \tau + n\varphi$ for some unique $n \in \mathbb{Z}$ and $\tau$. We set $f(z):=\frac{z}{|z|}e(\varphi)$ if $|z| \leq n^2$ and set $f(z):=-\frac{z}{|z|}e(\varphi)$ if $|z| > n^2$. It is easy to see that $f:\mathbb{C} \to \mathbb{C}$ is bounded and for every $\theta \in [0,1)$ we have:

$$\left \vert \Phi(\theta)-\Re \left(\overline{f(\rho \, e(\theta))} \; \cdot e(\varphi + \theta)\right)\right\vert \to 0, \; \rho \to +\infty$$

\noindent with $\Phi \equiv -1$. This convergence is evidently not uniform in $\theta \in [0,1)$. Also we take $\{y(n)\}_{n\geq1}:=\{0\}_{n\geq 1}$ to be a zero sequence.

Observe that all the conditions of theorem \ref{main_theorem}, except condition 1, are satisfied. We claim that for all $x(1) \in \mathbb{C}$ the solution $\{x(n)\}_{n\geq 1}$ of the equation $x(n+1) = e(\varphi)x(n)+f(x(n)) + y(n), n\geq 1$ is unbounded. 

Indeed, since always $x(2) \neq 0$, then without loss of generality we may assume that $x(1)\neq 0$. It is obvious from \eqref{solution_recurrence_relation} that 

$$\arg^*x(n) = \arg^{*}x(1) + (n-1)\varphi, n\geq 1$$ 

\noindent Hence, all the $\arg^*x(n)$ belong to the same orbit $\{\tau_0 + n \varphi : n \in \mathbb{Z}\}$ for some unique $\tau_0$. Thus, there exists some $n_0 \in \mathbb{Z}$ such that 

$$\arg^*x(n) = \tau_0 + (n_0+n)\varphi, n\geq 1$$ 
\noindent From \eqref{solution_recurrence_relation} and the construction of $f$ observe that:

\begin{equation}\label{counterexample_1_abs_value_of_solution}
    |x(n)| = |x(1)|+ \sum_{k=1}^{n-1}f(x(k)) = |x(1)|+\sum_{k=1}^{n-1}\left[\mathbbm{1}(|x(k)|\leq k^2) - \mathbbm{1}(|x(k)|>k^2\right]
\end{equation}

Because $|x(n+1)|-|x(n)| \in \{-1,1\}$ we observe that $|x(n)|\leq|x(1)|+n$. Therefore, for all $n$ large enough we have $|x(n)|\leq n^2$. Since 

$$\arg^*x(n) = \tau_0 + (n_0+n)\varphi, n\geq 1$$
 
\noindent we get $f(x(n)) = \frac{x(n)}{|x(n)|}e(\varphi)$ for all $n$ large enough. We conclude that there exists some integer $k_0$ such that for all $k \geq k_0$ we have $f(x(k)) = \frac{x(k)}{|x(k)|}e(\varphi)$ and, thus, from \eqref{counterexample_1_abs_value_of_solution} it follows that $|x(n)|\geq |x(1)|+n-2k_0, n\geq 1$ which implies that solution $\{x(n)\}_{n\geq 1}$ is unbounded.

\subsection{Riemann integrability of the limit function $\Phi$ can't be replaced with Lebesgue integrability}

Let $\varphi \in [0,1)$ be an arbitrary irrational number and as in the previous subsection we take $\{y(n)\}_{n\geq1}:=\{0\}_{n\geq 1}$ to be a zero sequence.

Every real $t \ \text{mod} \ 1$ can be represented as the decimal fraction $t=\overline{0.d_1 d_2 d_3 d_4 \ldots} \ \text{mod} \ 1$. For each positive integer $N$ we define the function $F_N:\mathbb{T} \to \mathbb{Q}$ as $F_N(t):=\overline{0.d_1 d_2\ldots d_N}, t=\overline{0.d_1d_2d_3\ldots} \in \mathbb{T}$. We denote $\varphi\mathbb{Q} := \{q\varphi \ \ \text{mod} \ 1 : q \in \mathbb{Q}\}$ and define the \text{"warped"} versions of functions $F_N$ namely the functions $F^{(\varphi)}_{N} : \mathbb{T} \to \varphi\mathbb{Q}$. $F^{(\varphi)}_{N}(t):=\varphi \cdot F_{N}(t/\varphi) \; \text{mod} \ 1$. Observe that for all $N$ we have

$$|t-F_N^{(\varphi)}(t)|= \varphi \cdot \left\vert\frac{t}{\varphi}-F_N\left(\frac{t}{\varphi}\right)\right\vert \leq \varphi \cdot 10^{-N+1}$$

\noindent and, thus, $\sup_{t \in \mathbb{T}}|t-F_N^{(\varphi)}(t)| \to 0, N \to \infty$.

Now we construct the function $f:\mathbb{C}\to \mathbb{C}$ as follows.

\begin{equation*}
    f(z):=\left\{
    \begin{array}{@{}rl@{}}
         1 & \text{if $z=0$} \\
         \frac{z}{|z|}e(\varphi) &  \text{if $0< |z| \leq 1$}\\
         \left\{-\frac{z}{|z|} + (|z|-1)\left[e\left(F_N^{(\varphi)}(\arg^{*}z)\right) - \frac{z}{|z|}\right]\right\} e(\varphi) & \text{if $\arg^*z \notin \varphi\mathbb{Q}$ and $|z| \in (N,N+1]$} \\ 
         & \text{for some $N \in \mathbb{N}$} \\

         \frac{z}{|z|}e(\varphi) & \text{if $\arg^*z \in \varphi\mathbb{Q}$}
    \end{array}
    \right .
\end{equation*}

If $z \in \mathbb{C}$ with $\arg^*z \in \varphi\mathbb{Q}$ or $0<|z| \leq 1$, then we see that $\arg^{*}(e(\varphi)z)=\arg^*{[e(\varphi)z + f(z)}]$ and $|e(\varphi)z)+f(z)|=|e(\varphi)z)|+1=|z|+1$.

If $z \in \mathbb{C}$ with both $\arg^*z \notin \varphi\mathbb{Q}$ and $|z| \in (N, N+1]$ for some $N \in \mathbb{N}$, then 

$$e(\varphi)z+f(z)=\left\{z  -\frac{z}{|z|} + (|z|-1)\left[e\left(F_N^{(\varphi)}(\arg^{*}z)\right) - \frac{z}{|z|}\right] \right\}e(\varphi) = $$

$$ = (|z|-1)\left\{ \frac{z}{|z|}+ \left[e\left(F_N^{(\varphi)}(\arg^{*}z)\right) - \frac{z}{|z|}\right] \right\}e(\varphi) = $$

$$ = (|z|-1) \cdot e\left(F_N^{(\varphi)}(\arg^{*}z)\right) \cdot  e(\varphi)$$

\noindent Thus, we conclude that $\arg^{*}(e(\varphi)z+f(z))=F_N^{(\varphi)}(\arg^{*}z) + \varphi \in \varphi\mathbb{Q}$ and $|e(\varphi)z+f(z)|=|z|-1$ in case $\arg^*z \notin \varphi\mathbb{Q}$, $|z| \in (N, N+1]$ for some $N \in \mathbb{N}$.

Now it is easy to see from \eqref{solution_recurrence_relation} that solution $\{x(n)\}_{n\geq1}$ of the equation $x(n+1)=e(\varphi)x(n)+f(x(n)) + y(n), n\geq1$ satisfies  $|x(n)|\geq n-2, n\geq 1$ for every $x(1) \in \mathbb{C}$, hence, all the solutions $\{x(n)\}_{n\geq1}$ are unbounded.

Since 

$$\left\vert e\left(F_N^{(\varphi)}(\arg^{*}z)\right) - \frac{z}{|z|} \right\vert = \left\vert e\left(F_N^{(\varphi)}(\arg^{*}z)\right) - e(\arg^{*}z) \right\vert \leq |F_N^{(\varphi)}(\arg^{*}z) - \arg^{*}z| \leq \varphi\cdot 10^{-N+1}$$

\noindent we see that

$$\sup_{\theta \in [0,1)} \left \vert \Phi(\theta)-\Re \left(\overline{f(\rho \, e(\theta))} \; \cdot e(\varphi + \theta)\right)\right\vert \to 0, \; \rho \to +\infty$$

\noindent with the limit function $\Phi : [0,1)\to \mathbb{R}$ satisfying $\Phi(\theta)=1$ if $\theta \in \varphi\mathbb{Q}$ and $\Phi(\theta)=-1$ if $\theta \notin \varphi\mathbb{Q}$. Set $\varphi\mathbb{Q} \subset \mathbb{T}$ has the Lebesgue measure zero, hence, $\Phi$ is Lebesgue integrable with $ \int_{[0,1)}\Phi(\theta)d\theta = -1$. On the other hand, evidently, $\Phi$ is not Riemann integrable.

All the conditions of theorem \ref{main_theorem}, except Riemann integrability of $\Phi$, are satisfied, however, as we have shown, all the solutions of the equation $x(n+1)=e(\varphi)x(n)+f(x(n)) + y(n), n\geq1$ are unbounded.

\subsection{In the condition 3 of theorem \ref{main_theorem} we can't take $\beta=0$}

In this subsection we give a construction, which shows that condition 3 of theorem \ref{main_theorem} can't be replaced with the following:

\vspace{10pt}

\noindent \text{$3^{\prime}.$} There exists some constant $D_0 \in \mathbb{N}$ such that for all $n_2>n_1\geq 1$ with $n_2 - n_1 \geq D_0$ we have 
$$\frac{1}{n_2 - n_1}\left\vert\sum_{n=n_1}^{n_2 - 1}y(n)e(-n\varphi)\right\vert + \int_{[0,1)} \Phi(s) ds < 0 $$ 

\vspace{10pt}

Let $\varphi \in [0,1)$ be an arbitrary irrational and let $h:(0,+\infty) \to (1, +\infty)$ be a strictly increasing function such that $h(t)$ diverges to $+\infty$ arbitrarily fast as $t \to +\infty$. Define the function $f:\mathbb{C} \to \mathbb{C}$ as follows: 

\begin{equation*}
    f(z):=\left\{
    \begin{array}{@{}rl@{}}
         0 & \text{if $z=0$} \\
         
        -\frac{z}{|z|}\cdot(1-\frac{1}{h(|z|)})e(\varphi) & \text{if $|z|\neq 0$}
         
    \end{array}
    \right .
\end{equation*}

\noindent and define the sequence $\{y(n)\}_{n\geq 1}:=\left\{\left(1-\frac{1}{n^2}\right)e(n\varphi)\right\}_{n\geq1}$. 

Obviously $f$ and $\{y(n)\}_{n\geq1}$ are bounded and observe that 
$$\sup_{\theta \in [0,1)} \left \vert \Phi(\theta)-\Re \left(\overline{f(\rho \, e(\theta))} \; \cdot e(\varphi + \theta)\right)\right\vert \to 0, \; \rho \to +\infty$$

\noindent with $\Phi \equiv -1$. It is easy to see that for all $n_2 \geq n_1 \geq 1$ we have

$$\frac{1}{n_2 - n_1}\left\vert\sum_{n=n_1}^{n_2 - 1}y(n)e(-n\varphi)\right\vert + \int_{[0,1)}\Phi(s)ds < 0$$

We claim that for every $x(1) \in \mathbb{C}$ the solution $\{x(n)\}_{n\geq1}$ of the equation $x(n+1)=e(\varphi)x(n) + f(x(n))+y(n), n\geq 1$ is unbounded. Take an arbitrary $x(1) \in \mathbb{C}$ and suppose for the sake of contradiction that there exists some $L > 0$ such that $|x(n)|\leq L \text{ for all }n\geq1$. Apply reverse triangle inequality to \eqref{solution_recurrence_relation} and obtain the following estimate

$$|x(n)|\geq \left\vert e(n\varphi)\sum_{j=1}^{n-1}y(j)e(-j\varphi) \right\vert - \left\vert e(n\varphi)x(1) + e(n\varphi)\sum_{j=1}^{n-1}f(x(j))e(-j\varphi) \right\vert = $$

$$ = \left\vert \sum_{j=1}^{n-1}y(j)e(-j\varphi) \right\vert - \left\vert x(1) + \sum_{j=1}^{n-1}f(x(j))e(-j\varphi) \right\vert$$

\noindent now apply the triangle inequality to the second term in the difference above and obtain

$$|x(n)|\geq  \left\vert \sum_{j=1}^{n-1}y(j)e(-j\varphi) \right\vert - |x(1)|-\sum_{j=1}^{n-1}|f(x(j))| $$

\noindent Since $|x(n)|\leq L, n\geq1$ by assumption, we have $|f(x(n))| \leq 1-\frac{1}{h(L)}, n\geq 1$. Therefore,

$$|x(n)| \geq \left\vert \sum_{j=1}^{n-1}y(j)e(-j\varphi) \right\vert - |x(1)|- (n-1) \left(1-\frac{1}{h(L)} \right) = $$

$$ = \left\vert \sum_{j=1}^{n-1}\left(1-\frac{1}{j^2}\right) \right\vert - |x(1)|- (n-1) \left(1-\frac{1}{h(L)} \right) = $$

$$ = (n-1) - \sum_{j=1}^{n-1}\frac{1}{j^2} - |x(1)| - (n-1)\left(1-\frac{1}{h(L)} \right) = $$

$$ = (n-1)\cdot\frac{1}{h(L)} - |x(1)| - \sum_{j=1}^{n-1}\frac{1}{j^2}$$

\noindent Since $\sum_{j=1}^{n-1}\frac{1}{j^2} \leq \frac{\pi^2}{6}$ we conclude that for all $n\geq 1$ we have

$$|x(n)|\geq (n-1)\cdot\frac{1}{h(L)} - |x(1)|-\frac{\pi^2}{6}$$

\noindent which contradicts our initial assumption $|x(n)|\leq L, n\geq1$. Therefore, all the solutions of the equation $x(n+1)=e(\varphi)x(n) + f(x(n))+y(n), n\geq 1$ are unbounded.

\vspace{10pt}

\noindent \textbf{Comment.} We would like to remark that in our construction we have 

$$\sup_{\theta \in [0,1)} \left \vert \Phi(\theta)-\Re \left(\overline{f(\rho \, e(\theta))} \; \cdot e(\varphi + \theta)\right)\right\vert = \frac{1}{h(\rho)} \ \text{ for }\rho > 0$$

\noindent It shows that all the solutions of the difference equation $x(n+1)=e(\varphi)x(n) + f(x(n))+y(n), n\geq 1$ can be unbounded no matter how fast the uniform convergence in the condition 1 of theorem \ref{main_theorem} is.

\begin{tabular}{@{}l@{}}%
    A. Chaikovskyi \\
    \textsc{Taras Shevchenko National University of Kyiv, Ukraine}\\
    \textit{E-mail address}: \texttt{andriichaikovskyi@knu.ua}
\end{tabular}

\vspace{10pt}

\begin{tabular}{@{}l@{}}%
    O. Liubimov \\
    \textsc{Taras Shevchenko National University of Kyiv, Ukraine}\\
    \textit{E-mail address}: \texttt{liubimov\_oleksandr@knu.ua}
\end{tabular}

\end{document}